\documentclass[12pt,reqno]{amsart}
\usepackage{a4wide}
\usepackage{amssymb}

\numberwithin{equation}{section}

\newtheorem{theorem}{Theorem}[section]

\newtheorem{lemma}[theorem]{Lemma}

\theoremstyle{definition}

\theoremstyle{remark}
\newtheorem{remark}[theorem]{Remark}

\newcommand{\al}{\alpha}
\newcommand{\be}{\beta}
\newcommand{\ep}{\varepsilon}
\newcommand{\ga}{\gamma}
\newcommand{\la}{\lambda}
\newcommand{\de}{\delta}
\newcommand{\om}{\omega}
\newcommand{\rd}{\mathbb{R}^d}
\newcommand{\hrd}{H^{\alpha}(\mathbb{R}^{d})}
\newcommand{\hprd}{H^{\alpha}_p(\mathbb{R}^d)}
\newcommand{\dhrd}{\dot{H}^{\alpha}(\mathbb{R}^{d})}
\newcommand{\xy}{\frac{1}{|x-y|^{\gamma}}}

\begin{document}
\title[ Nonlinear fractional Schr\"odinger equations ]
{Existence and stability of standing waves for nonlinear fractional Schr\"odinger equations
with Hartree type nonlinearity}

\author[D. Wu]
{Dan Wu}

\address{Dan Wu \newline
Academy of Mathematics and System Sciences, Beijing 100190, China.}

\email{danwu@amss.ac.cn}

\thanks{}

\begin{abstract}
In this paper, we consider the nonlinear fractional Schr\"odinger
equations with Hartree type nonlinearity. We obtain the existence of
standing waves by studying the related constrained minimization
problems via applying the concentration-compactness principle. By
symmetric decreasing rearrangements, we also show that the standing
waves, up to a translations and phases, are positive symmetric
nonincreasing functions. Moreover, we prove that the set of
minimizers is a stable set for the initial value problem of the
equations, that is, a solution whose initial data is near the set
will remain near it for all time.
\end{abstract}

\maketitle

\section{Introduction}

We consider the following fractional nonlinear Schr\"odinger equation with Hartree type nonlinearity
\begin{equation}\label{nls}
i\psi_t+(-\Delta)^{\al}\psi-(|\cdot|^{-\ga}*|\psi|^2)\psi=0,
\end{equation}
where $0< \alpha<1$, $0<\ga<2\al$ and $\psi(x,t)$ is a
complex-valued function on $\rd\times\mathbb R$, $d\geq 2$.
The fractional Laplacian $(-\Delta)^{\al}$ is a non-local operator defined as
\begin{equation}\label{fra_Lap}
\mathcal{F}[(-\Delta)^{\al}\psi](\xi)=|\xi|^{2\al}\mathcal{F}\psi(\xi),
\end{equation}
where the Fourier transform is given by
\begin{equation}
\mathcal{F}\psi(\xi)=\frac{1}{(2\pi)^{2d}}\int_{\rd}\psi(x)e^{-i\xi\cdot x}dx.
\end{equation}

The fractional Schr\"odinger equation plays
a significant role in the theory of fractional quantum mechanics.
It was formulated by N. Laskin \cite{laskin2000fractional}\cite{laskin2000frac:levy}\cite{laskin2002fractional}
as a result of extending the Feynman path integral
from the Brownian-like to L\'evy-like quantum mechanical paths.
The L\'evy processes, occuring widely in physics, chemistry and biology,
lead to equations with the fractional Laplacians which have been recently studied by
\cite{cabre2010nonlinear}\cite{guan2006reflected}\cite{valdinoci2009long}.
When $\al=\frac12$, NLS \eqref{nls} can be used to describe the dynamics of pseudo-relativistic boson stars in the mean-field limit, see \cite{frank2009ground}.
When $\al=1$, the L\'evy motion becomes Brownian motion and the fractional
Schr\"odinger equation turns to be the well-known classical nonlinear Schr\"odinger equation
which has been studied by many authors, see for instance
\cite{cazenave2003semilinear}\cite{cazenave1982orbital}\cite{lieb1977existence}\cite{lions1987solutions}.

Recently, the fractional nonlinear Schr\"odinger equations with
power type nonlinearity have been studied by
\cite{guo2008existence}\cite{guo2012existence}\cite{guo2010global}.
In this paper, we consider Hartree type nonlinearity. It has been
showed in \cite{cho2012cauchy} that the equation \eqref{nls} is
locally well-posed in $\hrd$ and globally well-posed under some
conditions. In view of the scaling invariance, we know that the
equation \eqref{nls} is mass-critical if $\ga=2\al$ and mass-subcritical
if $\ga<2\al$. For mass-critical case $\ga=2\al$, \cite{cho2012finite}
and \cite{cho2012profile} investigate the blowup phenomena of NLS
\eqref{nls} with radial data. The aim of this paper is to
investigate existence and  stability of standing waves of NLS
\eqref{nls} in mass-subcritical case.

A standing wave of NLS \eqref{nls} means a solution of the special
form $ \psi(x,t)=e^{i\om t}u(x) $, where $ \om\in\mathbb R$ is a frequency.
In order to study the existence and stability of standing waves
to NLS \eqref{nls}, we first look for $(\om,u)$ satisfying the stationary equation
\begin{equation}\label{gq}
(-\Delta)^{\al}u-(|\cdot|^{-\ga}*|u|^{2})u=\om u  \quad \mathrm{in}\  \rd ,
\end{equation}
where $u(x)$ is complex-valued. For studying the
existence of solutions to \eqref{gq}, by the variational method, we
can consider the following constrained minimization problem:
\begin{equation}\label{eq}
E_q :=\inf \{E(u);\ u \in \hrd ,\  M(u)=q\},
\end{equation}
where the mass is defined as
\begin{equation}\label{mass}
M(u)=\int_{\rd} |u(x)|^2 dx,
\end{equation}
and the energy is
\begin{equation}\label{energy}
E(u)=\frac12\int_{\rd} |(-\Delta)^{\frac\al2}u(x)|^2 dx
-\frac14\iint_{\rd\times\rd}\xy|u(x)|^{2}|u(y)|^2 dxdy.
\end{equation}
\begin{remark}
N. Laskin \cite{laskin2002fractional} showed the hermiticity of the fractional Schr\"odinger operator and established the conservation laws of the mass and the energy.
\end{remark}
We will denote the set of minimizers of problem \eqref{eq} by
\[
G_q:= \{u \in H^{\al} ;\ E(u)=E_{q},\  M(u)=q \}.
\]

Let $S$ denote the set of the symmetric decreasing functions in
$\hrd$, that is,
\begin{equation}\label{}
S=\{u\in\hrd;\ u\geq0\ \mathrm{and}\ u(x)\leq u(y)\ \mathrm{if}\
|x|\geq|y|\},
\end{equation}
and let
\begin{equation}\label{}
S^{'}=\{u\in\hrd;\ u(x-y)=v(y)\ \mathrm{a.e.\ for\ some}\ v\in S\
\mathrm{and}\ y\in\rd\},
\end{equation}
the set of translates (a.e.) of functions in $S$. Two functions $u$
and $v$ in $S^{'}$ are said to be equicentered if $u(x-y)=v(y)$ a.e.
for some $v\in S$ and $y\in\rd$.

Throughout this paper, we always denote
$\|\cdot\|=\|\cdot\|_{H^{\alpha}(\rd)}$ and
$\|\cdot\|_p=\|\cdot\|_{L^p(\rd)}$ for simplicity.

Our main results in this paper are the following.

\begin{theorem}[Existence of standing waves]\label{T:exi}
Let $d\ge 2$, $0<\al<1$, and $\ga<2\al$. If $\{u_n\}$ is a minimizing sequence of problem \eqref{eq},
then there exists a sequence $\{y_{n}\} \subset \rd$ such that
$\{u_n(\cdot-y_n)\}$ contents a convergent subsequence in $H^{\al}(\rd)$.
In particular, there exists a minimizer for problem \eqref{eq},
which implies $G_q$ is not a empty set.
Moreover, we have
\begin{equation}\label{min=0}
\lim_{n\rightarrow\infty}\inf_{g \in G_{q}}\|u_{n}-g\|=0.
\end{equation}
\end{theorem}

\begin{theorem}[]\label{T:pro}
The standing waves obtained in Theorem \ref{T:exi} satisfy the
following properties:
\begin{enumerate}
\item The standing waves are continuous,
in particular, $G_q\subset C^{[2\al],2\al-[2\al]}(\rd)$, where
$[2\al]$ is the integer part of $2\al$;
\item If $g\in G_q$, then $|g|\in G_q$ and $|g|>0$ on $\rd$;
\item The standing waves are symmetric decreasing after modified translations and phases,
that is, $G_q\subset\{u;\ e^{i\theta}u(x-y)=v(y)\ \mathrm{a.e.\ for\
some}\ v\in S,\ \theta\in \mathbb{R}\ \mathrm{and}\ y\in\rd\}$.
\end{enumerate}

\end{theorem}

\begin{theorem}[$\hrd$-stable]\label{T:sta}
Under the assumptions of Theorem \ref{T:exi},
the set $G_q$ is $\hrd$-stable with respect to NLS \eqref{nls}, that is,
for every $\ep>0$, there exists $\de>0$ such that
if $u \in C(\mathbb{R},\hrd)$ is a solution to NLS \eqref{nls}
with the initial data $u_0$ satisfying
\[
\inf_{g \in G_q}\|u_0-g\|<\de,
\]
then for all $t > 0$, we have
\[
\inf_{g \in G_q}\|u(\cdot, t)-g\|<\ep.
\]
\end{theorem}

\section{Preliminaries}\label{S:pre}

In this section, we will collect some results known in existing
literature, which will be used in our paper. To start with,  we
recall the definition of $\hrd$, which is the fractional order
Sobolev space defined as
\[
\hprd:=\{u:\rd \rightarrow \mathbb{C};\ u \in L^p \ and\
\mathcal{F}^{-1}[(1+|\xi|^{2})^{\frac{\alpha}{2}}\mathcal{F}u] \in
L^p\},
\]
whose norm is given by
\[
\|\cdot\|_{\al,p}=\|\mathcal{F}^{-1}[(1+|\xi|^{2})^{\frac{\alpha}{2}}\mathcal{F}u]\|_p.
\]
In particular, we write $\hrd=H^\al_2(\rd)$ for brevity.
The following lemma gives an equivalent norm that is quite useful.
\begin{lemma}
Let $0< \al <1$, the norm $\|\cdot\|_{\al,2}$ of $\hrd$ is
equivalent to
\begin{equation*}
\|\cdot\|=\|\mathcal{F}^{-1}[(1+|\xi|^{\al})\mathcal{F}\cdot]\|_{2}
=\|\cdot\|_2+\|\cdot\|_{\dot{H}^\al(\rd)}.
\end{equation*}
\end{lemma}
This result follows easily from the fundamental inequality
\begin{equation*}
1+|\xi|^\al\leq(1+|\xi|^2)^{\frac\al2}\leq C(1+|\xi|^\al),
\end{equation*}
and the definitions of $\|\cdot\|_{\al,2}$ and $\|\cdot\|$.

Next, we give a lemma, which is another definition of the fractional
Laplacian and will be frequently used later.
\begin{lemma}\label{L:fra_lap}
Let $0< \al <1$, and $u(x)$ be a function in the Schwartz class on
$\rd$, then the fractional Laplacian of $u$ has a pointwise
expression as
\begin{equation*}
(-\Delta)^\al u(x)=C_{d,\al}\mathrm{P.V.}\int_{\rd}\frac{u(x)-u(y)}{|x-y|^{d+2\al}}dy,
\end{equation*}
where P.V. means the Cauchy principal value on the integral and
$C_{d,\al}$ is some positive normalization constant.
\end{lemma}
The equivalence of two definitions of the fractional Laplacian can be proved by Riesz potential
and the Green's second identity. Here we omit the details.
\cite{valdinoci2009long} gives a simple proof.

The following inequality, which is due to G. H. Hardy,
will play a major role in the nonlinearity estimates.
\begin{lemma}[The Hardy's inequality]\label{L:hardy}
For $0<\ga<d$, we have
\[
\sup_{y \in \mathbb{R}^{d}}\int_{\mathbb{R}^{d}}\frac{|u(x)|^{2}}{|x-y|^{\gamma}}dx
\leq C\|u\|^{2},
\]
where the constant C depends on d and $\ga$.
\end{lemma}

The following commutator estimates was developed in \cite{guo2012existence}
using Kato and Ponce's result \cite{kato1988commutator}.
\begin{lemma}[Commutator estimates]\label{L:com_est}
If $0<\al<1$ and $f,g \in \mathcal{S}$, the Schwartz class, then the following holds:
\[
\|(-\Delta)^{\frac{\alpha}{2}}(fg)-f(-\Delta)^{\frac{\alpha}{2}}g\|_{2}
\leq C(\|\nabla f\|_{p_{1}}\|(-\Delta)^{\frac{\alpha-1}{2}}\|_{q_{1}}
+\|(-\Delta)^{\frac{\alpha}{2}}f\|_{p_{2}}\|g\|_{q_{2}}),
\]
where $q_{1},p_{2} \in [2,+\infty)$
and $\frac{1}{p_{1}}+\frac{1}{q_{1}}=\frac{1}{p_{2}}+\frac{1}{q_{2}}=\frac{1}{2}$.
\end{lemma}

\begin{lemma}[Fractional Rellich Compactness theorem]\label{L:rellich}
Let $0<\al<d$, $1\leq p<\frac d\al$, $1\leq q<\frac{dp}{d-\al p}$
and $\Omega$ is a bounded open set with smooth boundary.
Suppose $\{u_n\}$ is a sequence in $L^p(\rd)$ satisfying
\[
\int_{\rd}|(-\Delta+1)^{\frac\al2}u_n(x)|^pdx
\]
are uniformly bounded, then $\{u_n\}$ has a convergent subsequence in $L^q(\Omega)$.
\end{lemma}
Lemma \ref{L:rellich} can be found in H. Hajaiej \cite{hajaiej2011some}.

For any given  Borel set $A$ with finite Lebegue measure, we define
its symmetric rearrangement by
\begin{equation}\label{}
A^*=\{x;\ |x|<r\}\ \mathrm{with}\ \frac{|\mathbb{S}^{d-1}|}{d}r^d=\mathfrak{L}^d(A),
\end{equation}
where $|\mathbb{S}^{d-1}|$ is the surface area of the unit ball in $\rd$.
This allowed us to define the symmetric decreasing rearrangement
of a characteristic function of a set A as
\begin{equation}\label{}
\chi^*_A(x)=\chi_{A^*}(x)\ \mathrm{for}\ x\in\rd.
\end{equation}
Clearly $\chi^*_A\in S$ and $\|\chi^*_A\|_1=\|\chi_{A^*}\|_1=\mathfrak{L}^d(A)$.
Given $f\in\hrd$, we define
\begin{equation}\label{}
f^*(x)=\int^\infty_0\chi^*_{\{|f|>t\}}(x)dt,
\end{equation}
which has following properties:
\begin{itemize}
\item $f^*$ is radial, nonnegative and nonincreasing, i.e. $f^*\in S$;
\item If $p\in[1,\infty]$ and $f\in L^p(\rd)$, then
\begin{equation}\label{lp}
\|f^*\|_p=\|f\|_p.
\end{equation}
\end{itemize}
Moreover, from Theorem 2.1. in \cite{hajaiej2011some} we have
\begin{lemma}\label{L:hal}
Let $0\leq\al\leq1$, then we have
\begin{equation}
\|u^*(x)\|_{\dot{H}^\al(\rd)}\leq\|u(x)\|_{\dot{H}^\al(\rd)}.
\end{equation}
\end{lemma}

F. Riesz \cite{riesz1930inegalite} showed the following inequality.
For a recent account of the theorems, we refer the reader to \cite{lieb2001analysis}.
\begin{lemma}[Riesz's rearrangement inequality]\label{L:Riesz}
Let $f$, $g$ and $h$ be three nonnegative functions on $\rd$, Then we have
\begin{equation}\label{Riesz}
\iint_{\rd\times\rd}f(x)g(x-y)h(y)dxdy\leq \iint_{\rd\times\rd}f^*(x)g^*(x-y)h^*(y)dxdy.
\end{equation}
\end{lemma}
Furthermore, E.H. Lieb \cite{lieb1977existence} has established the strict version of
the Riesz's rearrangement inequality.
\begin{lemma}[Strict version of Lemma \ref{L:Riesz}]\label{L:Riesz_sv}
Under the assumptions of Lemma \ref{L:Riesz}.
If $g\in S$ and $g$ is  positive  and  strictly  decreasing, that is,
\begin{equation}\label{Riesz_sv}
0<g(x)<g(y)\ \mathrm{if}\ |x|>|y|,
\end{equation}
then \eqref{Riesz} is a strict inequality when the right hand side
is finite unless $f$ and $h$ are equicentered functions in $S^{'}$.
\end{lemma}
\begin{theorem}[Global existence of weak solutions for NLS \eqref{nls}]\label{T:global}
If $0<\al<1$, $\ga<2\al$ and $\psi_0\in\hrd$, then there exists a global
weak solution $\psi(x,t)\in C(\mathbb{R},\hrd)$ to the Cauchy
problem of nonlinear fractional Schr\"odinger equations \eqref{nls}
with the initial date $\psi(x,0)=\psi_0(x)$.
\end{theorem}
See \cite{cho2012cauchy} for more details.

\section{The proof of main results}

In this section we give proofs of our main results listed in the
first section. To begin with, we  solve the constrained minimization
problem \eqref{eq}. It is known that, in this kind of problem, the
main difficulty concerns with the lack of compactness of the
minimizing sequences $\{u_{n}\}$ for the problem. Indeed, two bad
scenarios possible are
\begin{itemize}
\item Vanishing $u_n\rightharpoonup 0$,
\item Dichotomy $u_n\rightharpoonup$ $u$ and $\|u\|^2_2\neq q$.
\end{itemize}
In order to rule out the above two cases and to show that the
infimum is achieved, we employ the concentration-compactness
principle developed by P.L. Lions. The best general reference about
this method are \cite{lions1984concentration1} and
\cite{lions1984concentration2}. First of all, we introduce the
L\'evy concentration function.
$$
Q_n(r):=\sup_{y \in \rd}\int_{B(y,r)}|u_{n}(x)|^{2}dx.
$$
Since $\{Q_n\}$ is locally of bounded total variation and uniformly
bounded, by the Helly's selection theorem, we can find a convergent
subsequence, denoted again by $\{Q_n\}$  such that there is a
nondecreasing function $Q(r)$ satisfying
\[
\lim_{n\rightarrow+\infty}Q_n(r)=Q(r), \ \mathrm{for\ all}\  r>0.
\]
Note that $0\leq Q_n(r)\leq q$, there exists $\be \in [0,q]$ such that
\begin{equation}\label{beta}
\lim_{r\rightarrow+\infty}Q(r)=\be.
\end{equation}

\begin{lemma}\label{L:eq_neg}
For every $q>0$, we have $-\infty<E_{q}<0$.
\end{lemma}

\begin{proof}
For given $u \in H^{\alpha}(\mathbb{R})$ with $\|u\|_{2}=q$, letting
$u_{\lambda}=\lambda^{\frac{1}{2}}u(\lambda^{\frac{1}{d}}x)$, we
then have $\|u_{\lambda}\|_{2}=q$. By the definition of energy, we
have
$$
E(u_{\la})=\frac12\la^{\frac{2\al}{d}}\int_{\rd} |(-\Delta)^{\alpha}u(x)|^2 dx
-\frac{1}{4}\la^{\frac{\ga}{d}}\iint_{\rd\times\rd}\xy|u(x)|^{2}|u(y)|^{2} dxdy.
$$
Since $0<\ga<2\al$, we can take  $\lambda>0$ sufficiently small such
that $E(u_{\la})<0$. Hence $E_{q}<E(u_{\la})<0$.

On the other hand,  Hardy's inequality implies
\begin{equation}
\begin{split}
\iint_{\rd\times\rd} \xy|u(x)|^{2}|u(y)|^{2} dxdy
&\leq\sup_{y \in \rd}\int_{\rd} \frac{1}{|x-y|^{\ga}}|u(x)|^2dx\|u\|^{2}_{2}\\
&\leq C\|u\|^{2}_{\dot{H}^{\frac{\ga}{2}}}\|u\|^{2}_{2}.
\end{split}
\end{equation}
Using  Sobolev's inequality and  Young's inequality, we deduce that
\begin{equation}\label{h_bdd}
\frac14H_\ga(u,u)\leq C\|u\|^{\frac\ga\al}_{\dot{H}^{\al}}\|u\|^{4-\frac{\ga}{\al}}_{2}
\leq \ep\|u\|^2_{\dot{H}^\al}
+C_\ep\|u\|^{\frac{8\al-2\ga}{2\al-\ga}},
\end{equation}
where $\varepsilon$ is a sufficiently small positive constant.
Hence, for $u \in \hrd$ with $\|u\|_{2}=q$ and sufficiently small $\ep$,
$$
E(u)\geq\frac12\|u\|^{2}-\frac{1}{2}q-\ep\|u\|^{2}-C_\ep q^{\frac{4\al-\ga}{2\al-\ga}}
\geq-\frac12q-C_\ep q^{\frac{4\al-\ga}{2\al-\ga}},
$$
which implies $E_q>-\infty$. So, $-\infty<E_{q}<0$.
\end{proof}

\begin{lemma}\label{L:van}
Vanishing does not occur, that is, $\be>0$, for every $q>0$.
\end{lemma}
To prove this lemma, we need the following two lemmas.

\begin{lemma}\label{L:min_bdd}
Every minimizing sequence $\{u_n\}$ for problem $\eqref{eq}$ is bounded in $\hrd$,
and there exists a constant $\delta>0$ such that $H_{\ga}(u_{n},u_{n})\geq\delta>0$ for sufficiently large n.
\end{lemma}
\begin{proof}
Firstly, it follows from \eqref{h_bdd} in Lemma \ref{L:eq_neg} that
$$
\frac14\iint_{\rd\times\rd}\xy|u_n(x)|^2|u_n(y)|^2dxdy
\leq \ep\|u_n\|^2+C_\ep\|u_n\|^{\frac{8\al-2\ga}{2\al-\ga}}_2.
$$
From this, we deduce that
\begin{equation}
\begin{split}
\frac12\|u_n\|^2&\leq E(u_n)+\frac12\|u_n\|^2_2+\frac14\iint_{\rd\times\rd}\xy|u_n(x)|^2|u_{n}(y)|^2dxdy\\
&\leq E(u_n)+\frac{1}{2}q+\ep\|u_n\|^2+C_\ep q^{\frac{4\al-\ga}{2\al-\ga}}
\end{split}
\end{equation}
Since $\{u_n\}$ is a minimizing sequence, we can get the result by
taking $\ep<\frac 12$.

For the second part, suppose that the lemma were false. Then we
could find subsequences $\{u_{n_k}\}$ such that
$$
\iint_{\rd\times\rd}\xy|u_{n_k}(x)|^2|u_{n_k}(y)|^2dxdy \rightarrow 0,\
as\ k\rightarrow+\infty.
$$
By the definition of energy, it follows immediately that
$$
E(u_{n_k})\rightarrow E_q\geq0,\
as\ k\rightarrow+\infty,
$$
which contradicts Lemma \ref{L:eq_neg}.
\end{proof}

\begin{lemma}\label{L:min_van}
Suppose $\{u_n\}$ is a minimizing sequence for the problem \eqref{eq} and satisfying
\[
\sup_{y\in\rd}\int_{B(y,r)}|u_n(x)|^2dx\rightarrow0,
\]
then, we have
\[
\iint_{\rd\times\rd}\xy|u_n(x)|^2|u_n(y)|^2dxdy\rightarrow0, \ as\ n\rightarrow\infty.
\]
\end{lemma}
\begin{proof}
Consider a minimizing sequence $\{u_n\}$ for problem \eqref{eq}.
For every $\ep>0$, since $\{u_n\}$ are bounded in $L^2$, we can find $r_\ep>0$ such that
\[
\iint_{|x-y|\geq r_\ep}\xy|u_n(x)|^2|u_n(y)|^2dxdy
\leq\frac{\ep}{2}.
\]
Next we divide the domain. For every positive $r$, we
can find countable balls $\{B(z_i,r)\}$ such that
\[
\rd\subset\bigcup^\infty_{i=1}B(z_i,r),
\]
and  every point in $\rd$ belongs to at most $d+1$ of these balls,
which implies
\begin{equation}
\sum^\infty_{i=1}\int_{B(z_i,r)}|u_n(x)|^2dx\leq (d+1)\|u_n\|^2_2.
\end{equation}
Consequently, if $x$ in some $B(z_i,r)$ and  $|x-y|\leq r_\ep$,
then there exists at most $N_\ep$ balls such that
\[
\{y \in \rd;\  |x-y|\leq r_\ep,\ x \in B(z_i,r) \}\subset\bigcup^{N_\ep}_{k=1}B(z_{i_{k}},r),
\]
where $N_\ep$ only depends on $\ep$. By the above facts, using
H\"older's and  Hardy's inequalities, we have
\begin{align*}
&\iint_{|x-y|\leq r_\ep}\xy|u_{n}(x)|^2|u_{n}(y)|^2dxdy\\
&\leq \sum^\infty_{i=1}\int_{B_x(z_i,r)}
\left[\sum^{N_\ep}_{k=1}\int_{B_y(z_{i_{k}},r)}\xy|u_{n}(x)|^2|u_{n}(y)|^2dy\right]dx\\
&\leq\sum^\infty_{i=1}\|u_{n}(x)\|^2_{L^2(B_x(z_i,r))}
\left[\sum^{N_\ep}_{k=1}\sup_{x \in B_x(z_i,r)}\int_{B_y(z_{i_{k}},r)}\xy|u_{n}(y)|^2dy\right]\\
&\leq\sum^\infty_{i=1}\|u_{n}(x)\|^2_{L^2(B_x(z_i,r))}
\sum^{N_\ep}_{k=1}\sup_{x \in B_x(z_i,r)}
\|u_{n}(x)\|^{2-\frac{2\ga}{\al}}_{L^2(B_y(z_{i_{k}},r))}
\left(\int_{B_y(z_{i_{k}},r)}\xy|u_{n}(y)|^2dy\right)^{\frac{\ga}{\al}}\\
&\leq CN_\ep\|u_{n}(x)\|^{\frac{\ga}{\al}}\left(\sum^\infty_{i=1}\|u_{n}(x)\|^2_{L^2(B_x(z_i,r))}\right)
\left(\sup_{y \in \mathbb{R}}\int_{B(y,r)}|u_{n}(x)|^{2}dx\right)^{1-\frac{\ga}{\al}}\\
&\leq C(d+1)N_\ep\|u_{n}\|^2_2\|u_{n}(x)\|^{\frac{\ga}{\al}}
\left(\sup_{y \in \mathbb{R}}\int_{B(y,r)}|u_{n}(x)|^{2}dx\right)^{1-\frac{\ga}{\al}}.
\end{align*}
Finally, taking $n$ to $\infty$, the second part can also be bounded
by $\frac{\ep}{2}$, which proves the lemma.
\end{proof}

\begin{proof}[Proof of Lemma \ref{L:van}]
Suppose, arguing by contradiction, that $\be=0$, then there exist a
positive $r_0$ and a subsequence $\{u_{n_k}\}$ of a minimizing
sequence $\{u_{n}\}$ such that
$$
\sup_{y \in \rd}\int_{B(y,r_0)}|u_{n_k}(x)|^{2}dx \rightarrow 0, \  as\ k\rightarrow\infty.
$$
Since $\{u_{n_k}\}$ is also a minimizing sequence, by Lemma \ref{L:min_van},
it follows that
$$
\iint_{\rd\times\rd}\xy|u_{n_k}(x)|^2|u_{n_k}(y)|^2dxdy
\rightarrow 0, \ as \ k\rightarrow\infty,
$$
which contradicts Lemma \ref{L:min_bdd}.
\end{proof}

\begin{lemma}\label{L:dic_1}
Let $q_1, q_2$ be positive real numbers, then
$E_{q_1+q_2}<E_{q_1}+E_{q_2}$.
\end{lemma}

\begin{proof}
Given $u \in \hrd$ with $\|u\|^2_2=q$, we let $u_\la(x)=\la^{\ga_1}u(\la^{\ga_2}x)$,
where $\ga_1=\frac{2\al-\ga+d}{8\al-2\ga}$ and $\ga_2=\frac{1}{2\al-\ga}$.
Then $\|u_\la\|^2_2=\la q$ and
\[
E(u_\la)=\la^{\frac{8\al-2\ga}{2\al-\ga}}E(u).
\]
Therefore,
\[
E_{\la q}=\inf_{\substack{u \in \hrd \\ M(u)=\la q}}E(u_\la)
=\la^{\frac{8\al-2\ga}{2\al-\ga}}\inf_{\substack{u \in \hrd \\ M(u)=q}}E(u)
=\la^{\frac{8\al-2\ga}{2\al-\ga}}E_q.
\]
According to Lemma \ref{L:eq_neg}, we know that $E_q$ is negative for all $q>0$.
For $\frac{8\al-2\ga}{2\al-\ga}>1$, it follows easily from Jensen's inequality that
\[
E_{q_1+q_2}=(q_1+q_2)^{\frac{8\al-2\ga}{2\al-\ga}}E_1
<(q_1^{\frac{8\al-2\ga}{2\al-\ga}}+q_2^{\frac{8\al-2\ga}{2\al-\ga}})E_1
=E_{q_1}+E_{q_2}.
\]
\end{proof}

\begin{lemma}\label{L:dic_2}
Suppose $0<\be<q$, then $E_\be+E_{q-\be}\leq E_q$.
\end{lemma}

\begin{proof}
For every $\ep>0$, there exists $r_\ep>0$ such that
\begin{equation}\label{rep}
\iint_{|x-y|\geq r_\ep}\xy|u_{n}(x)|^2|u_{n}(y)|^2dxdy
\leq\frac{\ep}{2},
\end{equation}
and
\[
\be-\frac{\ep}{4}<Q(r_\ep)\leq Q(3r_\ep)\leq \be.
\]
Then there exists $N_\ep \in \mathbb{N}^+$ such that for every $n\geq N_\ep$, we have
\[
\be-\frac{\ep}{2}<Q_n(r_\ep)\leq Q_n(3r_\ep)< \be+\frac{\ep}{2}.
\]
Next we choose $\{y_n\}\subset\rd$ so that
\begin{equation}\label{ring_bbd}
\be-\ep<\int_{B(y_n,r_\ep)}|u_n(x)|^2dx\leq\int_{B(y_n,3r_\ep)}|u_n(x)|^2dx<\be+\ep.
\end{equation}
Now let us define $\phi_r(x)=\phi(\frac xr)$ and $\tilde{\phi}_r(x)=\tilde{\phi}(\frac xr)$,
where $\phi \in C^\infty_0(B(0,2))$ is a smooth cutoff function satisfying
\begin{equation}\label{vw}
0\leq\phi(x)\leq1\ \mathrm{and}\
\phi(x)=\left\{
\begin{aligned}& 1,&&|x|\leq 1, \\
&0,&&|x|\geq2,
\end{aligned}
\right.
\end{equation}
and $\tilde{\phi}(x)=1-\phi(x)$.
With this notation, we write
\begin{equation}
\begin{split}
v_n(x)&=\phi_r(x-y_n)u_n(x),\\
w_n(x)&=\tilde{\phi}_r(x-y_n)u_n(x).
\end{split}
\end{equation}
It follows immediately
\begin{equation}
\begin{split}
\be-\ep<&\int_{\rd}|v_n(x)|^2dx< \be+\ep,\\
q-\be-\ep<&\int_{\rd}|w_n(x)|^2dx< q-\be+\ep.
\end{split}
\end{equation}
The conclusion follows if
\begin{equation}\label{reduce}
E(v_n)+E(w_n)\leq E(u_n)+c\ep,
\end{equation}
for some positive constant $c$.

To see this, note that from \eqref{vw}, there exist $\mu_n, \nu_n
\in [1-\ep,1+\ep]$ such that
\[
\|\mu_n v_n\|^2_2=\be \ \mathrm{and} \ \|\mu_n w_n\|^2_2=q-\be.
\]
we therefore deduce that
\begin{align*}
E_\be&\leq E(\mu_n v_n)\leq E(v_n)+c\ep,\\
E_{q-\be}&\leq E(\mu_n w_n)\leq E(w_n)+c\ep,
\end{align*}
for some positive constant $c$ independent of $\ep$.
Combining the above two inequalities and using (\ref{reduce}), we have
\[
E_\be+E_{q-\be}\leq E(v_n)+E(w_n)+c\ep\leq E(u_n)+c\ep.
\]
Passing to the limit, we can then prove the Lemma.

To sum up, what is left is to show (\ref{reduce}).
According to the definitions of $v_n$ and $w_n$, we have
\begin{align*}
E(v_n)+E(w_n)&=\frac12\int_{\rd}|(-\Delta)^{\frac{\alpha}{2}}v_n(x)|^2 dx
+\frac12\int_{\rd}|(-\Delta)^{\frac{\alpha}{2}}w_n(x)|^2dx\\
&-\frac14\iint_{\rd\times\rd}\xy\left[|v_n(x)|^2|v_n(y)|^2+|w_n(x)|^2|w_n(y)|^2\right]dxdy.
\end{align*}
Applying Lemma \ref{L:com_est} and using the Sobolev's inequalities, we obtain
\begin{align*}
\int_{\rd}|(-\Delta)^{\frac{\alpha}{2}}&[\phi_r(x-y_n)u_n(x)]|^2 dx
-\int_{\rd}\phi^2_r(x-y_n)|(-\Delta)^{\frac{\alpha}{2}}u_n(x)|^2 dx\\
&\leq C(\|\nabla\phi_r\|_{\frac{d}{1-\al}}
\|(-\Delta)^{\frac{\alpha-1}{2}}u_n\|_{\frac{2d}{d-2(\al-1)}}
+\|(-\Delta)^{\frac{\alpha}{2}}\phi_r\|_{2+\frac{d}{\al}}\|u_n\|_{2+\frac{4\al}{d}})\\
&\leq C\left(\frac{1}{r^d}\|\nabla\phi\|_{\frac{d}{1-\al}}\|u_n\|_{2}
+\frac{1}{r^{\frac{2\al^2}{2\al+d}}}
\|(-\Delta)^{\frac{\alpha}{2}}\phi\|_{2+\frac{d}{\al}}\|u_n\|\right).
\end{align*}
After taking $r$ larger enough, we derive from the above inequality that
\[
\int_{\rd}|(-\Delta)^{\frac{\alpha}{2}}[\phi_r(x-y_n)u_n(x)]|^2 dx
\leq\int_{\rd}\phi^2_r(x-y_n)|(-\Delta)^{\frac{\alpha}{2}}u_n(x)|^2 dx+c\ep.
\]
In the same way, we  see that
\[
\int_{\rd}|(-\Delta)^{\frac{\alpha}{2}}[\tilde{\phi}_r(x-y_n)u_n(x)]|^2 dx
\leq\int_{\rd}\tilde{\phi}^2_r(x-y_n)|(-\Delta)^{\frac{\alpha}{2}}u_n(x)|^2 dx+c\ep.
\]
Recalling $0\leq\phi,\tilde{\phi}\leq1$, we conclude from the above two inequalities that
\[
\int_{\rd}|(-\Delta)^{\frac{\alpha}{2}}v_n(x)|^2 dx
+\int_{\rd}|(-\Delta)^{\frac{\alpha}{2}}w_n(x)|^2dx
\leq\int_{\rd}|(-\Delta)^{\frac{\alpha}{2}}u_n(x)|^2 dx+c\ep.
\]
Now it remains to prove
\begin{equation}\label{non}
\iint_{\rd\times\rd}\xy\left[
|u_n(x)|^2|u_n(y)|^2-|v_n(x)|^2|v_n(y)|^2-|w_n(x)|^2|w_n(y)|^2\right]dxdy\leq c\ep
\end{equation}
Expanding the left hand side of \eqref{non} and combining the
equivalent terms, we have
\begin{equation}\label{begin}
\begin{split}
2\iint_{\rd\times\rd}\xy(
&|v_n(x)|^2|w_n(y)|^2+2|v_n(x)||w_n(x)||v_n(y)|^2\\
+2&|v_n(x)||w_n(x)||w_n(y)|^2+2|v_n(x)||v_n(y)||w_n(x)||w_n(y)| )dxdy
\end{split}
\end{equation}
Indeed, except the first term $|v_n(x)|^2|w_n(y)|^2$,
the remainders are integral on the ring $B(y_n,2r_\ep)\setminus B(y_n,r_\ep)$
in $\rd_x$ or $\rd_y$ (or both).
Therefore, from \eqref{ring_bbd}, we have
\begin{equation}
\begin{split}
\iint_{\rd\times\rd}&\xy|v_n(x)||w_n(x)||v_n(y)|^2dxdy\\
&\leq\sup_{x \in\rd}\int_{B_y(y_n,2r)}\xy|v_n(y)|^2dy
\int_{B_x(y_n,2r+r_\ep)\setminus B_x(y_n,r)}|v_n(x)||w_n(x)|dy\\
&\leq\|u_n\|\int_{B(y_n,2r_\ep)\setminus B(y_n,r_\ep)}|u_n(x)|^2dx\leq c\ep
\end{split}
\end{equation}
Similarly,
\begin{equation}
\begin{split}
\iint_{\rd\times\rd}\xy(|v_n(x)||w_n(x)||w_n(y)|^2+|v_n(x)||v_n(y)||w_n(x)||w_n(y)|)dxdy
\leq c\ep
\end{split}
\end{equation}
To estimate the first term, recalling \eqref{rep},
we only need to deal with the integral
on the set $\{(x,y) \in \rd\times\rd;\ |x-y|\leq r_\ep\}$.
Similar arguments as above imply that
\begin{equation}\label{end}
\begin{split}
\iint_{|x-y|\leq r_\ep}&\xy|v_n(x)|^2|w_n(y)|^2dxdy\\
&\leq \sup_{y \in\rd}\int_{B_x(y_n,2r_\ep)}\xy|v_n(x)|^2dx
\int_{B_x(y_n,3r_\ep)\setminus B_x(y_n,r_\ep)}|w_n(y)|^2dy\\
&\leq\|u_n\|\int_{B(y_n,3r_\ep)\setminus B(y_n,r_\ep)}|u_n(x)|^2dx\leq c\ep
\end{split}
\end{equation}
From \eqref{begin}-\eqref{end}, we proved \eqref{non}. This finishes the proof.

\end{proof}

\begin{proof}[Proof of Theorem \ref{T:exi}]
Recalling the definition of $\be$ in \eqref{beta}, by Lemma
\ref{L:van}, Lemma \ref{L:dic_1} and Lemma \ref{L:dic_2}, we know
that every minimizing sequence $\{u_n\}$ for problem \eqref{eq} has
a subsequence, denoted again by $\{u_n\}$, satisfying
\begin{equation}\label{compact}
\lim_{r\rightarrow+\infty}\lim_{n\rightarrow+\infty}\sup_{y\in\rd}\int_{B(y,r)}|u_n(x)|^2dx
=\be=q,
\end{equation}
which implies that for every positive $\ep>0$, there exist
$r_\ep>0$, $n_\ep\in\mathbb{N}^+$ and $\{y_n\}\subset\rd$ such that
for each $n>n_\ep$ and $r>r_\ep$,
\begin{equation}\label{compact}
\int_{B(y_n,r)}|u_n(x)|^2dx>q-\ep.
\end{equation}

According to Lemma \ref{L:min_bdd}, $\{u_n(\cdot-y_n)\}$ is bounded
in $\hrd$. By choosing subsequence if necessary, there exists
$g\in\hrd$ such that,
\[
u_n(\cdot-y_n)\rightharpoonup g\ \mathrm{weekly\ in}\ \hrd.
\]
We can find $R_\ep>r_\ep$ such that $\|g\|_{L^2(\rd\setminus B(0,R_\ep))}<\frac\ep2$.
Furthermore, by Lemma \ref{L:rellich}, there exists $N_\ep\in\mathbb{N}^+$ with $N_\ep>n_\ep$ such that
for $n>N_\ep$, we have
\[
\|u_n(\cdot-y_n)-g\|_{L^2(B(0,R_\ep))}<\frac\ep2.
\]
It follows immediately from the above that
\begin{equation}
\begin{split}
\|g\|_2&\geq\|u_n\|_2-\|u_n(\cdot-y_n)-g\|_{L^2(B(0,r_\ep))}
-\|u_n(\cdot-y_n)-g\|_{L^2(\rd\setminus B(0,r_\ep))}\\
&\geq \|u_n\|_{L^2(B(y_n,r_\ep)}-\|u_n(\cdot-y_n)-g\|_{L^2(B(0,r_\ep))}
-\|g\|_{L^2(\rd\setminus B(0,r_\ep))}\\
&\geq \sqrt{q-\ep}-\ep,
\end{split}
\end{equation}
which implies, by passing to the limit, $\|g\|^2_2\geq q$. On the
other hand, the weak lower semi-continuous deduces
\begin{equation}
q\leq\|g\|^2_2\leq\liminf_{n\rightarrow+\infty}\|u_n\|^2_2=q.
\end{equation}
Therefore, $\|g\|^2_2= q$, and consequently,
\[
u_n(\cdot-y_n)\rightarrow g\ \mathrm{strongly\ in}\ L^2(\rd),
\]
since $\{u_n(\cdot-y_n)\}$ converges weakly in $\hrd$.
Moreover, we have
\begin{equation}\label{lim_1}
\begin{split}
\iint_{\rd\times\rd}\xy&(|u_n(x-y_n)|^2|u_n(y-y_n)|^2-|g(x)|^2|g(y)|^2)dxdy\\
&\leq\iint_{\rd\times\rd}\xy[|u_n(x-y_n)|^2(|u_n(y-y_n)|^2-|g(y)|^2)\\
&\qquad +|g(y)|^2(|u_n(x-y_n)|^2-|g(x)|^2)]dxdy\\
&\leq C(\|u_n\|+\|g\|)\||u_n(\cdot-y_n)|^2-|g|^2\|_2\\
&\leq C(\|u_n\|+\|g\|)(\|u_n\|_2+\|g\|_2)\|u_n(\cdot-y_n)-g\|_2\rightarrow0,
\end{split}
\end{equation}
as $n\rightarrow\infty$.
Applying the weak lower semi-continuous again, we deduce that
\begin{equation}\label{lim_2}
\|g\|_{\dhrd}\leq\liminf_{n\rightarrow+\infty}\|u_n\|_{\dhrd}.
\end{equation}
From \eqref{lim_1} and \eqref{lim_2}, it follows immediately that
\begin{equation}
\be\leq E(g)\leq\liminf_{n\rightarrow+\infty}E(u_n)=\be.
\end{equation}
Hence, $ g$ is a minimizer of problem \eqref{eq} and
\begin{equation}
u_n(\cdot-y_n)\rightarrow g\ in\ \hrd.
\end{equation}
Next, arguing by contradiction, we prove \eqref{min=0}. Assume that
there exist $\ep_0>0$ and a subsequence $\{u_{n_k}\}$ of $\{u_n\}$
such that
\begin{equation}\label{contra}
\inf_{g\in G_q}\|u_{n_k}-g\|\geq\ep_0>0.
\end{equation}
From what has already been proved, we know that there exist a
subsequence of $\{u_{n_k}\}$, denoted again by $\{u_{n_k}\}$, and
$\{y_{n_k}\}\in\rd$ such that
\[
u_{n_k}(\cdot-y_{n_k})\rightharpoonup g\ in\ \hrd.
\]
Since $g(\cdot+y_{n_k})\in G_q$, it follows that
\[
\|u_{n_k}-g(\cdot+y_{n_k})\|=\|u_{n_k}(\cdot-y_{n_k})-g\|\rightarrow0,
\]
which contradicts \eqref{contra}.
\end{proof}

\begin{proof}[Proof of Theorem \ref{T:pro}]
Consider a minimizer $g\in G_q$. Assume that $g\in L^p(\rd)$, then
$(|\cdot|^{-\ga}*|g|^{2})g\in L^p(\rd)$. Note that  equation
\eqref{gq} can be written in the form
\begin{equation}\label{gq_four}
\mathcal{F}^{-1}[(1+|\xi|^{2\al})\mathcal{F}g]=(|\cdot|^{-\ga}*|g|^{2})g.
\end{equation}
It follows that $g\in H^{2\al,p}(\rd)$.
By Sobolev's embedding theorem, we have
\begin{equation}\label{induct}
g\in L^q(\rd)\ \mathrm{for\ all}\ \frac1q\in\left[\frac1p-\frac{2\al}{d},\frac1p\right].
\end{equation}
Consider the sequence $\{q_i\}$ defined by
\begin{equation*}
q_0=2\ \mathrm{and}\ q_{i+1}=\frac{dq_i}{d-2\al q_i}\ \mathrm{for}\ i\in\mathbb{N}^+
\end{equation*}
Since
\begin{equation*}
\frac{1}{q_{i+1}}-\frac{1}{q_i}=-\frac{2\al}{d}<0,
\end{equation*}
we deduce that $\frac{1}{q_i}\rightarrow-\infty$ as $i\rightarrow+\infty$,
so there exist $i_0\in\mathbb{N}^+$ such that
\begin{equation*}
\frac{1}{q_i}>0\ \mathrm{for}\ 0\leq i\leq i_0\
\mathrm{and}\ \frac{1}{q_{i_0+1}}\leq0.
\end{equation*}
By an induction argument and \eqref{induct},
it is not difficulty to show that $g\in L^{q_{i_0}}(\rd)$.
Applying once again \eqref{induct}, we deduce that
\begin{equation}
g\in L^q(\rd)\ \mathrm{for\ all}\
\frac1q\in\left[\frac{1}{q_{i_0+1}},\frac{1}{q_{i_0}}\right].
\end{equation}
In particular, we can take $q=\infty$, so that $g\in L^2(\rd)\bigcap
L^\infty(\rd)$. We obtain immediately that
$(|\cdot|^{-\ga}*|g|^{2})g\in L^2(\rd)\bigcap L^\infty(\rd)$. Using
\eqref{gq_four} again, we have $g\in H^{2\al}_p(\rd)$ for all $ p\in
[2,\,\infty)$. By Sobolev's embedding, $g\in
C^{[2\al],2\al-[2\al]}(\rd)$.

We now turn to the part (2) of Theorem \ref{T:pro}. Arguing by
contradiction, we assume that $|g|\notin S^{'}$. Then, Lemma
\ref{L:hal} gives us
\begin{equation}\label{rad_hal}
\|g^*\|_{\dot{H}^\al(\rd)}\leq\|g\|_{\dot{H}^\al(\rd)}.
\end{equation}
Since $\frac{1}{|x|^{-\ga}} \in S$ and satisfies \eqref{Riesz_sv},
it follows immediately from Lemma \ref{L:Riesz_sv} that
\begin{equation}\label{rad_non}
\iint_{\rd\times\rd}\xy|g(x)|^2|g(y)|^2dxdy
<\iint_{\rd\times\rd}\xy|g^*(x)|^2|g^*(y)|^2dxdy,
\end{equation}
unless $|g|\in S^{'}$ for some $y\in\rd$. Recalling \eqref{lp} and
combining \eqref{rad_hal} and \eqref{rad_non}, we conclude that
\[
\|g^*\|_2=\|g\|_2=q\ \mathrm{and}\ E(g^*)<E(g)=E_q,
\]
which contradicts the definition of $E_q$. Hence, we have proved
$|g|\in S^{'}$. Furthermore $|g|\in G_q$, since $E(|g|)=E(g^*)=q$.

Next we claim that $|g|(x)>0$ for all $x\in\rd$. To this end, we
arguing by contradiction. Suppose that there exists $x_0\in\rd$ such
that $|g|(x_0)=0$. Then, it follows from the equation \eqref{gq}
that $(-\Delta)^\al|g|(x_0)=0$. By Lemma \ref{L:fra_lap}, we have
\begin{equation*}
(-\Delta)^\al |g|(x_0)=C_{d,\al}
\left(\lim_{\ep\rightarrow0}\int_{\ep\leq|x_0-y|\leq r)}\frac{-|g|(y)}{|x_0-y|^{d+2\al}}dy+
\int_{\rd\setminus B(x_0,r)}\frac{-|g|(y)}{|x_0-y|^{d+2\al}}dy\right)=0,
\end{equation*}
which implies
\begin{equation*}
\int_{\rd\setminus B(x_0,r)}\frac{|g|(y)}{|x_0-y|^{d+2\al}}dy=0\ \mathrm{for\ all}\ r>0,
\end{equation*}
Therefore, $|g|\equiv0$ on $\rd$, which contradicts $|g|\in G_q$.
Let $u=\Re g$, $v=\Im g$, we have $g=u+iv$.
It follows that
\begin{align}
(-\Delta)^{\al}u-(|\cdot|^{-\ga}*|g|^{2})u&=\om u\\
(-\Delta)^{\al}v-(|\cdot|^{-\ga}*|g|^{2})v&=\om v,
\end{align}
Repeating the similar argument as before for the linear equation
\begin{equation}\label{linear}
(-\Delta)^{\al}h-(|\cdot|^{-\ga}*|g|^{2})h=\om h,
\end{equation}
we know that $u,v$ are continuous and $|u|,|v|>0$, Therefore, $u$
and $v$ both have constant signs. We claim that there exists
constants $\al,\be$ such that $u=\al|g|$ and $v=\be|g|$. If this
were not the case, there would exist a constant $c$ such that
$w=u-c|g|$ takes both positive and negative values. It is easy to
see that $w$ also satisfies the equation \eqref{linear}, which is a
contradiction. Likewise, the same conclusion is true for $v$. Thus,
we have proved that $g=\al|g|+i\be|g|=e^{i\theta}|g|$, where
$\theta$ is a constant satisfying $\tan(\theta)=\frac{\be}{\al}$.
Thus we have finished our proof of Theorem \ref{T:pro}.
\end{proof}

\begin{remark}
Actually we can prove the existence of radial standing waves of
the equation \eqref{nls} in a much simple way
by symmetric decreasing rearrangements of minimizing sequence in Theorem \ref{T:exi}.
In this way, however, we may not exclude the possibility of nonradial standing waves of the equation \eqref{nls} and could not deduce \eqref{min=0} in Theorem \ref{T:exi},
which plays a key role in the proof of Theorem \ref{T:sta}.
\end{remark}

\begin{proof}[Proof of Theorem \ref{T:sta}]
We arguing by way of contradiction. Suppose that the set $G_q$ is
not $\hrd$-stable. Then there exist $\ep_0>0$, a sequence
$\{u^{(0)}_m\}$ in $\hrd$ and $t_m\in\mathbb{R}$ such that
\begin{equation}\label{sta_1}
\inf_{g\in G_q}\|u^{(0)}_m-g\|<\frac1m,
\end{equation}
and
\begin{equation}\label{sta_2}
\inf_{g\in G_q}\|u_m(t_m)-g\|\geq\ep_0,
\end{equation}
where $u_m\in C(\mathbb{R},\hrd)$ are solutions to NLS \eqref{nls}
with initial date $u_m(x,0)=u^{(0)}_m(x)$. From \eqref{sta_1}, we
have as $m\rightarrow +\infty$
\[
\|u^{(0)}_m\|^2_2\rightarrow q,\qquad\,\, E(u^{(0)}_m)\rightarrow
E_q.
\]
Hence, we can find $\{\mu_m\}\subset\mathbb{R}$, satisfying
$\|\mu_mu^{(0)}_m\|^2_2=q$ and $\mu_m\rightarrow1$, such that
$\{\mu_mu^{(0)}_m\}$ is a minimizing sequence for the problem
\eqref{eq}. By the conservation laws,
\begin{gather*}
\|\mu_mu_m(t_m)\|^2_2=\|\mu_mu^{(0)}_m\|^2_2=q,\\
E(\mu_mu_m(t_m))\rightarrow E_q,
\end{gather*}
we know that $\{\mu_mu_m(t_m)\}$ is also a minimizing sequence for
the problem \eqref{eq}. According to Theorem \ref{T:exi}, there
exist a subsequence $\{\mu_{m_k}u_{m_k}(t_{m_k})\}$ of
$\{\mu_mu_m(t_m)\}$, and $\{g_{m_k}\}$ in $G_q$ such that
\begin{equation}\label{sta_2}
\|\mu_{m_k}u_{m_k}(t_{m_k})-g_{m_k}\|<\frac{\ep_0}{2},
\end{equation}
for sufficiently large $m_k$.
It follows that
\begin{equation*}
\begin{split}
\ep\leq\|u_{m_k}(t_{m_k})-g_{m_k}\|
&\leq\|u_{m_k}(t_{m_k})-\mu_{m_k}u_{m_k}(t_{m_k})\|+\|\mu_{m_k}u_{m_k}(t_{m_k})-g_{m_k}\|\\
&\leq|\mu_{m_k}-1|\|u_{m_k}(t_{m_k})\|+\frac{\ep_0}{2},
\end{split}
\end{equation*}
which leads to a contradiction since $\mu_m\rightarrow1$. Therefore,
the set $G_q$ is $\hrd$-stable with respect to NLS \eqref{nls}.
\end{proof}

\section*{Acknowledgements}
The author wishes to express his gratitude to Prof. Daomin Cao and Prof. Pigong Han
for several helpful comments and for many stimulating conversations.

\end{document}